\documentclass[12pt,reqno]{amsart}
\usepackage{color}
\usepackage[active]{srcltx}
\usepackage{a4wide}
\usepackage{amssymb, amsmath, mathtools}
\usepackage{graphicx}
\usepackage{float}
\usepackage[ansinew]{inputenc}
\usepackage{hyperref}
\usepackage{amsmath,amsthm,amsfonts,amssymb,amscd,overpic,mathrsfs}

\theoremstyle{plain}
\newtheorem{theorem}{Theorem}[section]
\newtheorem{maintheorem}{Theorem}
\newtheorem{lemma}[theorem]{Lemma}
\newtheorem{proposition}[theorem]{Proposition}

\newtheorem{maincorollary}[maintheorem]{Corollary}

\theoremstyle{remark}

\numberwithin{equation}{section}

\begin{document}
\title[Equilibrium states for a class of skew-products]{Equilibrium states for a class of skew-products}

\author[M. Carvalho]{Maria Carvalho}
\address{Maria Carvalho\\ Centro de Matem\'{a}tica da Universidade do Porto\\ Rua do
Campo Alegre 687\\ 4169-007 Porto\\ Portugal}
\email{mpcarval@fc.up.pt}

\author[S. A. P\'{e}rez]{Sebasti\'{a}n A. P\'{e}rez}
\address{Sebasti\'{a}n P\'{e}rez Opazo \\ Centro de Matem\'{a}tica da Universidade do Porto\\ Rua do
Campo Alegre 687\\ 4169-007 Porto\\ Portugal}
\email{sebastian.opazo@fc.up.pt}

\date{\today}
\thanks{MC and SP were partially supported by CMUP (UID/MAT/00144/2013) which is funded by FCT (Portugal) with national (MEC) and European structural funds, through the programs FEDER, under the partnership agreement PT2020.
SP also acknowledges financial support from a postdoctoral grant of the project PTDC/MAT-CAL/3884/2014.}
\keywords{Skew-product; Partial hyperbolicity; Dominated splitting; Equilibrium state.}
\subjclass[2010]{Primary 37D35, 
37A35, 
37D30. 
Secondary 37A05, 37A30. 
}

\begin{abstract}
We consider skew-products on $M \times \mathbb{T}^2$, where $M$ is the two-sphere or the two-torus, which are partially hyperbolic and semi-conjugate to an Axiom A diffeomorphism. This class of dynamics includes the open sets of $\Omega$-non-stable systems introduced by Abraham, Smale and Shub. We present sufficient conditions, both on the skew-products and the potentials, for the existence and uniqueness of equilibrium states, and discuss their statistical stability. 
\end{abstract}

\maketitle

\setcounter{tocdepth}{2}


\section{Introduction}

In a recently disclosed notebook\footnote{https://bowen.pims.math.ca/}, Bowen listed several topics and open problems worthwhile discussing, including the statistical properties of non-Axiom A systems. Some readers of those remarkable pages have emphasized the questions regarding the existence and uniqueness of equilibrium states, having in mind the complete thermodynamic formalism Bowen established in \cite{Bo1975} for Axiom A diffeomorphisms and H\"{o}lder continuous potentials, besides Bowen's axiomatic description of some systems with unique equilibrium states \cite{Bo1974}. The two main ingredients in the latter work, suggested by Smale's Spectral Decomposition Theorem for uniformly hyperbolic systems, are expansiveness and specification. Several weaker forms of these properties have since been introduced in connection with measures of maximal entropy (cf. \cite{P2016} and references therein). A short time ago, Climenhaga and Thompson \cite{CT2016} developed a new criterium to prove existence and uniqueness of equilibrium states for systems exhibiting non-uniform versions of expansiveness and specification, with respect to H\"{o}lder potentials subject to further technical conditions. Although powerful, to apply the main theorem of \cite{CT2016} to particular cases beyond hyperbolicity one has to cope with these additional demands on the dynamics and the potentials. Climenhaga, Fisher and Thompson succeeded in using this machinery to settle a thermodynamic formalism for both Bonatti-Viana diffeomorphisms \cite{BV2000, CFT-BV2018} and the family of partially hyperbolic, robustly transitive, derived from Anosov systems introduced by Mañé \cite{M1978, CFT-M2018}. One of the key ideas of these two articles is the reformulation of those supplementary requests on the dynamics in terms of the $C^0$-closeness to an Anosov system and the denseness of each center-stable/center-unstable foliation.

The purpose of this paper is to address these ergodic questions for Abraham-Smale \cite{AS1970} and Shub \cite{S1971, HPS1977} examples, the ones that paved the way to the proof of the nongenericity of $\Omega$-stable systems. Those examples are skew-products obtained through small local $C^0$-perturbations of a product $\sigma \times L$, where $\sigma \colon M \,\to\, M$ is either a Smale's horseshoe map on the two-dimensional sphere $M =\mathbb{S}^2$ (in the case of Abraham-Smale work) or an Anosov diffeomorphism of the two-torus $M = \mathbb{T}^2$ with two fixed points (the choice of Shub), and $L \colon: \mathbb{T}^2 \,\to\, \mathbb{T}^2$ is a linear hyperbolic toral automorphism. Yet, unlike Bonatti-Viana and Mañé examples, in Abraham-Smale and Shub's dynamics there are points whose center-stable foliation is not dense, nor are their unstable manifolds. The denseness of those foliations in Bonatti-Viana and Mañé examples was essential to ensure specification at a small scale in both \cite{CFT-BV2018} and \cite{CFT-M2018}. Therefore, we looked for another strategy.

Abraham-Smale and Shub's examples are homotopic to $\sigma \times L$ and partially hyperbolic, admitting a splitting with a one-dimensional, normally hyperbolic, non-trivial center subbundle. Besides, Abraham-Smale's diffeomorphism is the limit of hyperbolic subshifts of finite type in the sense of \cite{Y1981}; Shub's is robustly transitive \cite{HPS1977}. Their non-wandering sets are equal to $\Omega(\sigma) \times \mathbb{T}^2$ (cf. \cite{Y1981, HPS1977}), and no diffeomorphism $C^1$-close to them is $\Omega$-stable. Benefitting from the fact that they are factors of an Axiom A system, and making strong use of the one-dimensionality of the center bundle, Newhouse and L.-S. Young \cite{NY1983} proved that those systems have unique measures of maximal entropy, whose entropy is equal to $h_{\text{top}}(\sigma) + h_{\text{top}}(L)$. So, we decided to endeavor in generalizing their argument to a broader class of H\"{o}lder potentials.

Starting with an Axiom A diffeomorphism, we have considered partially hyperbolic skew-products that are close and homotopic to that original Axiom A system. A parameterized version of Franks' result \cite{F1969}, due to C. Robinson and included in \cite[Lemma 1]{NY1983}, then states that there is a semi-conjugacy back to the Axiom A transformation which, under additional requirements, is an almost conjugacy (that is, a topological conjugacy after neglecting a small invariant set). The leading assumption in this reasoning regards the presence of a large invariant region in $M \times \mathbb{T}^2$ within which one detects asymptotic contraction in average along the center foliation. Next, in order to lift the uniqueness of the equilibrium state of the Axiom A system to the skew-product, we extended to the pressure Ledrappier-Walters' estimates of the metric entropy of factors \cite{LW1977}. This was enough to deduce uniqueness of the equilibrium states for an open class of H\"{o}lder potentials with respect to Abraham-Smale and Shub's examples. Concerning existence of equilibrium states for continuous potentials, we just took into account that Abraham-Smale and Shub's diffeomorphisms are entropy-expansive \cite{CY2005}, so their entropy maps are upper semi-continuous \cite{Mz1976}. Furthermore, in view of the results in \cite{HPS1977} (with noteworthy generalizations in \cite{IlyNeg2012} and \cite{BKR2014}), these ergodic properties of Abraham-Smale and Shub's examples are also valid for small perturbations thereof, although those $C^1$-close diffeomorphims may not be skew-products. 

We remark that this line of argument has likewise inspired \cite{BFSV2012}, \cite{BF2013} and \cite{U2012} in connection with intrinsic ergodicity. We also observe that the properties proved in this work are still valid if, instead of $M \times \mathbb{T}^2$, we consider similar dynamical systems on $M \times \mathbb{T}^n$ with $n > 2$, since the stable and unstable directions of a partially hyperbolic diffeomorphism are integrable (see \cite{HPS1977}) and so is the center bundle, without restrictions on its dimension, for tori diffeomorphisms isotopic to a linear Anosov automorphism along a path of partially hyperbolic diffeomorphisms (cf. \cite{FRS2014}).

The paper is organized as follows. In Section~\ref{se:setting} we enumerate the essential properties requested from the skew-products we will work with. For the reader's convenience, a few definitions are recalled in Section~\ref{se:definitions}. The statements of the main results may be read in Section~\ref{se:main-results}, where we also give an outline of their proofs. After establishing sufficient conditions for the uniqueness of the equilibrium states in Sections~\ref{se:proof-Th-A}, \ref{se:proof-Th-B} and \ref{se:proof-Cor-C}, we analyze in Section~\ref{se:perturbations} the robustness of those conditions in the $C^1$-topology, and prove in Section~\ref{se:stability} the equilibrium states' stability with respect to both the dynamics and the potentials.

\section{Setting}\label{se:setting}

Let $\sigma \colon M\,\to \,M$ be a homeomorphism of a compact metric space $M$. Take a linear hyperbolic automorphism $A: \mathbb{R}^2 \, \to \, \mathbb{R}^2$ such that $A(\mathbb{Z}^2) \subset \mathbb{Z}^2$, $\mathrm{det} \, A = \pm 1$ and $A$ has a real eigenvalue of multiplicity one whose absolute value is strictly smaller than $1$. Let $\mathbb{T}^2=\mathbb{R}^2 \diagup \mathbb{Z}^2$ be the $2$-torus and $L: \mathbb{T}^2 \, \to \mathbb{T}^2$ the Anosov diffeomorphism induced by $A$.

Given a compact, connected manifold $Z$ without boundary, denote by $\text{Diff}^r(Z)$ the space of $C^r$ diffeomorphisms of $Z$ endowed with the norm of the uniform $C^r$-convergence. Consider a family of $C^1$ diffeomorphisms $(f_x)_{x \,\in\, M}$ acting on $\mathbb{T}^2$, and the skew-product of $\sigma$ and $(f_x)_{x \,\in\, M}$ on $M \times \mathbb{T}^2$ defined by
\begin{eqnarray}\label{eq:skew-product}
F \colon \, M \times \mathbb{T}^2 \quad &\to& \quad  M \times \mathbb{T}^2 \nonumber \\
(x,y) \quad &\mapsto& \quad \big(\sigma(x), \, f_x(y)\big).
\end{eqnarray}
As in \cite{NY1983}, we assume that $F$ satisfies the following conditions:
\begin{enumerate}
\item The map $x \in M \, \to \, f_x \in \text{Diff}^1(\mathbb{T}^2)$ is continuous.
\medskip
\item $F$ is homotopic to $\sigma \times L$ as a bundle map with trivial fiber $\mathbb{T}^2$, that is, there exists a continuous map $\mathcal{G} \colon M \times \mathbb{T}^2 \times [0,1] \, \to \, \mathbb{T}^2$ satisfying $\mathcal{G}(x,y,0)=f_x(y)$ and $\mathcal{G}(x,y,1)=L(y)$, for every $(x,y) \in M \times \mathbb{T}^2$.
\medskip
\item There is a one-dimensional lamination 
    $\mathfrak{F}$ of $M \times \mathbb{T}^2$ which is $F$-invariant and normally expanded. More precisely, for each $(x,y) \in M \times \mathbb{T}^2$, the leaf $\mathfrak{F}_{(x,y)}$ through $(x,y)$ is a smooth immersed line in $\{x\} \times \mathbb{T}^2$ and  $F(\mathfrak{F}_{(x,y)}) = \mathfrak{F}_{F(x,y)}$. Besides, there is a splitting $E^{\mathrm s}_{(x,y)} \oplus E^{\mathrm u}_{(x,y)}$ of the tangent space at $(x,y)$ to the fiber $\{x\} \times \mathbb{T}^2$ which varies continuously with $(x,y)$ and such that\\
\begin{itemize}
\item $D_y\,f_x (E^{\mathrm u}_{(x,y)}) = E^{\mathrm u}_{F(x,y)}$ and $D_y\,f_x (E^{\mathrm s}_{(x,y)}) = E^{\mathrm s}_{F(x,y)}$;
\medskip
\item $E^{\mathrm s}_{(x,y)}$ is the tangent space at $(x,y)$ of $\mathfrak{F}_{(x,y)}$;
\medskip
\item for each $x \in M$, there is a Riemannian metric on $\{x\} \times \mathbb{T}^2$ with induced norm $\|\cdot\|$ such that
$$\hspace{3cm} \inf_{\substack{v \,\in \,E^{\mathrm u}_{(x,y)} \,\,\,\colon\,\, \|v\|=1 \\ (x,\,y) \,\in \,M \times \mathbb{T}^2}}\,\,\|D_y\,f_x (v)\| \, > \max\Big\lbrace 1, \,\,\sup_{\substack{w \,\in \,E^{\mathrm s}_{(x,y)} \,\colon\, \|w\|=1 \\ (x,\,y) \,\in \,M \times \mathbb{T}^2}}\,\,\|D_y\,f_x (w)\|\Big\rbrace.$$
\end{itemize}
\end{enumerate}

\medskip

Under the previous assumptions, Lemma 1 in \cite{NY1983} (see also \cite{F1969}) ensures the existence of a semi-conjugacy between $\sigma \times L$ and $F$.

\begin{lemma}\cite[Lemma 1]{NY1983}\label{le:semi-conjugacy}
There is a continuous surjective map $H \colon M \times \mathbb{T}^2 \to M \times \mathbb{T}^2$ of the form $H(x,y)=\left(x, H_x(y)\right)$ such that $(\sigma \times L) \circ H = H \circ F$ and each $H_x\colon \mathbb{T}^2 \, \to \, \mathbb{T}^2$ is homotopic to the identity.
\end{lemma}

\section{Preliminaries}\label{se:definitions}

In this section we will recall some background material and preliminary useful results.

\subsection{Almost conjugacy}\label{sse:almost-conjugacy}

Given two measures spaces $(X, \mathfrak{A}, \eta)$ and $(Y, \mathfrak{B}, \gamma)$, we say that the measure preserving transformations $\alpha: X \, \to \,X$ and $\beta: Y \, \to \, Y$ are \emph{almost conjugate} if there are sets $X_0 \in \mathfrak{A}$ and $Y_0 \in \mathfrak{B}$ such that
\begin{itemize}
\item[(a)] $\alpha(X_0) = X_0$, $\,\eta(X_0) = 0$, $\,\beta(Y_0) = Y_0$ and $\,\gamma(Y_0) = 0$;
\medskip
\item[(b)] $\alpha_{|_{{X \,\setminus \,X_0}}}$ is topologically conjugate to $\beta_{|_{{Y \,\setminus \,Y_0}}}$.
\end{itemize}


\subsection{Equilibrium states}\label{sse:pressure}
Let $f \colon X \,\to\, X$ be a continuous map on a compact metric space $(X,d)$, $C^0(X)$ the Banach algebra of real-valued continuous functions of X equipped with the supremum norm and $\mathcal{S}(X)$ the collection of all the subsets of $X$. For $\varphi \in C^0(X)$ and $n \in \mathbb{N}$, we denote $\sum_{j=0}^{n-1}\,\varphi(f^j(x))$ by $S_n(f,\varphi)(x)$. Given a subset $Y \subset X$, $\varepsilon > 0$ and a natural number $n$, a subset $E$ of $Y$ is said to be $(n,\varepsilon)$-\emph{separated} with respect to $f$ if for any $x,y \in E$ with $x \neq y$ there is some $j \in \{0,\cdots,n-1\}$ such that $d\,(f^j(x), \,f^j(y)) > \varepsilon$. The \emph{topological pressure} of $f$ is the operator
$$P_{\text{top}}(f\,,\cdot\, ,\cdot\,) \colon \,\,C^0(X) \times \mathcal{S}(X) \, \to \, \mathbb{R}\cup \{+\infty\}$$
which assigns to each $\varphi \in C^0(X)$ and each $Y \subset X$ the value
$$P_{\text{top}}(f, \varphi, Y) = \lim_{\varepsilon \rightarrow 0}\,P(f, \varphi, Y, \varepsilon)$$
where 
\begin{align*}
P(f, \varphi, Y, \varepsilon)\, &= \,\limsup_{n \rightarrow +\infty}\,\frac{1}{n}\,\log\,P(f, \varphi, Y, \varepsilon, n) \\
P(f, \varphi, Y, \varepsilon, n)\, &= \, \sup\,\,\Big\{\sum_{x \,\in\, E}\, e^{S_n(f,\,\varphi)(x)} \colon \,E\, \text{ is a }\,(n,\varepsilon)\text{-separated subset of $Y$}\Big\}.
\end{align*}
We often write $P_{\text{top}}(f, \varphi)$ in place of $P_{\text{top}}(f, \varphi, X)$. The topological pressure generalizes the notion of topological entropy in the sense that, when $\varphi \equiv  0$, then  $P_{\text{top}}(f, 0, Y) = h_{\text{top}}(f, Y)$. We also write $h_{\text{top}}(f)$ instead of $h_{\text{top}}(f, X)$. 

Let $\mathscr{P}(X)$ be the set of Borel probability measures on $X$ endowed with the weak$^*$-topology, $\mathscr{P}(X,f)$ be the subset of $f$-invariant elements of $\mathscr{P}(X)$ and $\mathscr{P}_e(X,f)$ be its subset of ergodic measures. The thermodynamic formalism identifies distinguished elements in $\mathscr{P}(X,f)$, called \emph{equilibrium states}: these are probability measures that maximize the quantity $h_\mu(f) + \int \varphi\,d\mu$, where $\varphi$ is a fixed potential in $C^0(X)$ and $h_\mu(f)$ denotes the metric entropy of $f$ with respect to $\mu$ (definition in \cite[\S4]{W1981}). An important motivation for the search of such measures comes from the Variational Principle \cite[Theorem 9.10]{W1981}, which states that, for every potential map $\varphi \in C^0(X)$, the supremum of the functional $\mu \in \mathscr{P}(X,f) \to h_\mu(f) + \int \varphi\,d\mu$ coincides with the supremum of its restriction to $\mathscr{P}_e(X,f)$, and is equal to the topological pressure $P_{\text{top}}(f, \varphi)$. In what follows, $\mathscr{P}(f, \varphi)$ stands for the set of equilibrium states of the pair $(f, \varphi)$.

\subsection{Entropy-expansiveness}\label{sse:entropy-exp}

Bowen introduced the notion of entropy-expansive map in \cite{Bo1972} referring to systems which seem to be \emph{expansive} with regard to entropy. More precisely, a continuous invertible map $f \colon X \,\to\, X$ on a compact metric space $(X,d)$ is expansive if there exists $\varepsilon > 0$ such that, for every $x \in X$,
$$B^\infty_\varepsilon(x) := \big\{y \in X \colon \, d(f^k(x), \, f^k(y)) < \varepsilon \quad \forall \, k \in \mathbb{Z}\big\} = \{x\}.$$
The transformation $f$ is said to be \emph{entropy-expansive} if there exists $\varepsilon > 0$ such that
$$\sup_{x \, \in \, X}\,\,h_{\text{top}}(f, B^\infty_\varepsilon(x)) = 0.$$
It is known that if $f$ is entropy-expansive, then the entropy map
$$\mu \in \mathscr{P}(X,f) \quad \to \quad h_\mu(f)$$
is upper semi-continuous \cite{Mz1976}, and so $\mathscr{P}(f, \varphi) \neq \emptyset$ for every $\varphi \in C^0(X)$.

\subsection{Partial hyperbolicity}\label{sse:partial-hyp}

Let $M$ be a compact Riemannian manifold and a diffeomorphism $f \colon M \,\to \,M$. An invariant compact set $K \subset M$ is said to be \emph{partially hyperbolic} by $f$ if the tangent bundle above $K$ admits a $Df$-invariant splitting
$E^{\mathrm s}(f) \oplus E^{\mathrm c}(f) \oplus E^{\mathrm u}(f)$ such that $E^{\mathrm s}$ is uniformly contracted and $E^{\mathrm u}$ is uniformly expanded, and the possible contraction and expansion of $Df$ in $E^c(f)$ are weaker than those in the complementary bundles. More precisely, there exist constants $N \in \mathbb{N}$ and $\lambda > 1$ such that, for every $x \in K$ and every unit vector $v^\ast\in E^{\ast}(x,f)$, where $\ast=\mathrm s,\mathrm c, \mathrm u$, we have
\medskip
\begin{itemize}
\item[(a)] $\quad \lambda\, \|Df_x^N (v^{\mathrm s})\| <  \|Df_x^N (v^{\mathrm c})\| < \lambda^{-1}\,\|Df_x^N (v^{\mathrm u})\|$
\medskip
\item[(b)] $\quad \|Df_x^N (v^{\mathrm s})\| < \lambda^{-1} < \lambda < \|Df_x^N (v^{\mathrm u})\|$.
\medskip
\end{itemize}

Generalizing the definition of hyperbolicity, the partially hyperbolic diffeomeorphisms are those systems whose non-wandering sets decompose into finitely many invariant transitive pieces, each of them being partially hyperbolic by $f$. We will call these subsets \emph{basic pieces}. For comprehensive surveys on this subject, we refer the reader to 
\cite{BDV2005} and \cite{HP2006}.

Partial hyperbolicity is a robust property, and a partially hyperbolic diffeomorphism $f$ admits stable and unstable foliations, say $W^{\mathrm s}(f)$ and $W^{\mathrm u}(f)$, which are $f$-invariant and tangent to $E^{\mathrm s}(f)$ and $E^{\mathrm u}(f)$, respectively \cite{BDV2005}. By contrast, the center bundle $E^{\mathrm c}(f)$ may not have a corresponding tangent foliation, and the same problem befalls either $E^{\mathrm s}(f)\oplus E^{\mathrm c}(f)$ or $E^{\mathrm c}(f)\oplus E^{\mathrm u}(f)$. When these exist (as happens with the open set of systems we consider here) we denote them by $W^{\mathrm c}(f)$, $W^{\mathrm{cs}}(f)$ and $W^{\mathrm{cu}}(f)$, and refer to them as the center, center-stable and center-unstable foliations, respectively.

Partial hyperbolic diffeomorphisms with a splitting exhibiting a one-dimensional center bundle are entropy-expansive (cf. \cite{CY2005}; see also \cite{DF2011, DFPV2012} for generalizations), and so they admit at least one equilibrium state for every continuous potential. However, even if we assume that the system is topologically mixing, there may be, for instance, more than one measure of maximal entropy \cite{HHTU2012}.

\section{Main results}\label{se:main-results}

Our first result compares the topological pressure operators associated to $\sigma \times L$ and to a skew-product as the ones described in Section~\ref{se:setting}; additionally, it establishes sufficient conditions for the existence of equilibrium states associated to the skew-products described in Section~\ref{se:setting}.

\begin{maintheorem}\label{teo:maintheorem-1}
Let $F \colon M \times \mathbb{T}^2 \, \to \, M \times \mathbb{T}^2$ be a skew-product as described in Section~\ref{se:setting} and $H$ be the semi-conjugacy between $F$ and $\sigma \times L$ given by Lemma~\ref{le:semi-conjugacy}. Then, for every potential $\varphi \in C^0(M \times \mathbb{T}^2)$ the following properties hold:
\begin{enumerate}
\item $P_{\mathrm{top}}(F, \, \varphi \circ H) = P_{\mathrm{top}}(\sigma \times L, \, \varphi)$.
\medskip
\item Any equilibrium state of $F$ and $\varphi \circ H$ projects by $H_*$ to an equilibrium state of $\sigma \times L$ and $\varphi$.
\medskip
\item If, in addition, $\sigma \times L$ has an equilibrium state $\nu_\varphi$ for $\varphi$, and the set
$$\mathcal{A} := \Big\{H(x,y) \in M \times \mathbb{T}^2 \colon \, H^{-1}(H(x,y)) = \{(x,y)\}\Big\}$$
satisfies $\nu_\varphi \,(\mathcal{A}) = 1$, then $F$ has an equilibrium state for the potential $\varphi \circ H$.
\end{enumerate}
\end{maintheorem}

\medskip

The existence of equilibrium states for $F$ depends not only on the potential map, but also on the dynamics of $\sigma \times L$. And, as $L$ is an Anosov diffeomorphism, the ergodic properties of $\sigma \times L$ strongly relies on the underlying dynamics $\sigma$. For instance, if $\sigma$ has no probability measure of maximal entropy, as the examples described in \cite[\S8]{W1981} or \cite{Mz1973}, then the same happens with $F$ (cf. \cite[Theorem 1]{NY1983}). Yet, if $\sigma$ is expansive, then there exists at least one equilibrium state for $F$ and every continuous potential. Indeed, in this case, the set of points that prevent expansiveness must be contained in a compact subset of the center laminations, inside curves whose lengths remain uniformly bounded after forward iterates of $F$ (more details on \cite[page 619]{NY1983}) and so $F$ is entropy-expansive (cf. \cite[Proposition 6]{CY2005}). Thus, if $\sigma$ is expansive, then skew-product $F$ is entropy-expansive.

The uniqueness of equilibrium states is, in general, a harder problem. 
We remark that, besides the existence, Theorem~\ref{teo:maintheorem-1} (3) also provides uniqueness if $\sigma\times L$ only has one equilibrium state associated to $\varphi$.  So, we need to find an additional condition which leads to the main assumption of this item of Theorem~\ref{teo:maintheorem-1}.
Our next result establishes one such a condition,
summoning the dynamical features of the skew-products $F$, 
 to prove that the equilibrium state of the pair $(F, \, \varphi \circ H)$ is unique for some class of potentials $\varphi$. Let us examine that extra condition, suggested by \cite{NY1983}. For each $(x,y) \in  M \times \mathbb{T}^2$, define the \textit{lower Lyapunov exponent} $\chi(x,y)$ of $F^{-1}$ at $(x,y)$ along the $\mathfrak{F}$-direction by
$$\chi(x,y):= \liminf_{n \, \to \, + \infty} \, \frac{1}{n} \, \log \|D_{(x,y)} F^{-n}|_{E^{\mathrm s}_{(x,y)}}\|$$
where $D_{(x,y)} F^{-n}$ denotes the derivative $D_{y} (\pi_2 \circ F^{-n})$ and $\pi_2:M \times \mathbb{T}^2\to  \mathbb{T}^2$ is the canonical projection on the second factor. The condition $\chi(x,y)>0$ roughly states that $F^{-n}$ expands lengths within the leaf $\mathfrak{F}(x,y)$ for large enough $n \in \mathbb{N}$. This happens if, for instance, for a large subset of elements $x \in M$ the diffeomorphism at the fibre  $f_{x}:\{x\}\times \mathbb{T}^2 \to \{\sigma(x)\}\times \mathbb{T}^2$ is Anosov (identifying $\{x\}\times \mathbb{T}^2$ and $\{\sigma(x)\}\times \mathbb{T}^2$ with $\mathbb{T}^2$) with stable foliation equal to $\mathfrak{F}$. Consider the set
$$\mathcal{E} = \big\{(x,y) \in M \times \mathbb{T}^2 : \quad \chi(x,y)>0 \big\}.$$
If there exists a $\sigma$-invariant set $B \subset M$ such that $B \times \mathbb{T}^2\subset \mathcal{E}$, then Lemma 4 of \cite{NY1983} shows that the restriction $H \colon B \times \mathbb{T}^2 \, \to \, B \times \mathbb{T}^2$ of the semi-conjugacy $H$ is a conjugacy between $F_{|_{B \times \mathbb{T}^2}}$ and $(\sigma \times L)_{|_{B \times \mathbb{T}^2}}$. This is the way we will explore to relate the equilibrium states of $\sigma \times L$ with the ones of $F$.

\begin{maintheorem}\label{teo:maintheorem-2} Let $F \colon M \times \mathbb{T}^2 \, \to \, M \times \mathbb{T}^2$ be a skew-product as described in Section~\ref{se:setting}, and $H$ be the semi-conjugacy between $F$ and $\sigma\times L$ given by Lemma~\ref{le:semi-conjugacy}. Assume that there is a $\sigma$-invariant set $B \subset M$ such that $B \times \mathbb{T}^2 \subset \mathcal{E}$. If $\sigma \times L$ has a unique equilibrium state $\nu_\varphi$ associated to the potential $\varphi \in C^0(M \times \mathbb{T}^2)$ and $\nu_\varphi(B \times \mathbb{T}^2)=1$, then $F$ has a unique equilibrium state $\mu_{\varphi \circ H}$ for the potential $\varphi \circ H \in C^0(B \times \mathbb{T}^2)$. Furthermore, the systems $(\sigma \times L, \, \nu_\varphi)$ and $(F,\, \mu_{\varphi\circ H})$ are measure theoretically isomorphic.
\end{maintheorem}

We now apply Theorem~\ref{teo:maintheorem-1} and Theorem~\ref{teo:maintheorem-2} to a non-hyperbolic setting. Let $\sigma: M \to M$ be a $C^1$ diffeomorphism on a compact connected Riemannian manifold $M$ satisfying Smale's Spectral Decomposition Theorem (\cite[Theorem 6.2]{S1967}). Using the techniques of \cite[Section 8]{HPS1977}, we may find a smooth diffeomorphism $F:M \times \mathbb{T}^2 \to M \times \mathbb{T}^2$ homotopic to $\sigma \times L$ such that its non-wandering set $\Omega(F)$ is $\Omega(\sigma) \times \mathbb{T}^2$ and $F$ satisfies the conditions (1)-(3) of Section~\ref{se:setting}. Note that, if $\Lambda$ is a basic set of $\sigma$, then so is $\Lambda \times \mathbb{T}^2$ for $\sigma \times L$, and by \cite{Bo1975} there exists a unique equilibrium state $\nu_\varphi$ for each H\"{o}lder potential $\varphi$ with respect to the dynamics $(\sigma \times L)_{|_{\Lambda \times \mathbb{T}^2}}$. Besides, $F$ admits a partially hyperbolic splitting of the fiber bundle at $\Lambda \times \mathbb{T}^2$ with a one-dimensional center bundle tangent to $\mathfrak{F}$. Therefore, 
under the hypotheses of Theorem~\ref{teo:maintheorem-1} and Theorem~\ref{teo:maintheorem-2}, the equilibrium state $\mu_{\varphi \circ H}$ is unique and inherits all ergodic properties of $\nu_\varphi$.

An important class of non-hyperbolic diffeomorphisms obtained through this process contains the well known examples of Abraham-Smale \cite{AS1970} and Shub \cite{HPS1977}. The skew-product studied by Abraham and Smale, namely $F_{AS}$, is a homotopic deformation on $\mathbb{S}^2 \times \mathbb{T}^2$ of the direct product of a Smale's horseshoe $\sigma: \mathbb{S}^2\to \mathbb{S}^2$ by a linear Anosov diffeomorphism. The one analyzed by Shub, which we denote by $F_{Sh}$, is obtained deforming on $\mathbb{T}^2 \times \mathbb{T}^2$ the product of two Anosov diffeomorphisms (not necessarily linear). The global dynamics of these examples satisfy the following properties:
\newpage
\begin{itemize}
\item[(a)] No diffeomorphism $C^1$-close to $F_{AS}$ (respectively, $F_{Sh}$)   is $\Omega$-stable.
\medskip
\item[(b)] The non-wandering set of $F_{AS}$ (respectively, $F_{Sh}$) is homeomorphic to the non-wandering set of any diffeomorphism $C^1$-nearby.
    \medskip
\item[(c)] The spectral decomposition of $F_{AS}$ is given by $\Omega(F_{AS})= \Omega_1 \cup \Omega_2 \cup \Omega_3$, where $\Omega_1 = \{p_0\} \times \mathbb{T}^2$, $\Omega_2 = \Lambda \times \mathbb{T}^2$ and $\Omega_3 = \{p_1\} \times \mathbb{T}^2$, and $p_0$ is a sink, $p_1$ is a source and $\Lambda$ is a horseshoe of the dynamics $\sigma \colon\, \mathbb{S}^2 \to \mathbb{S}^2$. The transitive set $\Omega_2$ is not hyperbolic due to the presence of two saddles with different stable indices. However, $\Omega_2$ admits a partially hyperbolic splitting.
    \medskip
\item[(d)] In Shub's example, the non-wandering set is (robustly) transitive, so $\Omega(F_{Sh}) = \mathbb{T}^2 \times \mathbb{T}^2$. Again, this set is not hyperbolic but has a partially hyperbolic splitting.
\medskip
\end{itemize}
In both examples, the corresponding family $x \in M \to f_x \in \mathrm{Diff}^1(\mathbb{T}^2)$ changes continuously from an Anosov diffeomorphism to a derived from Anosov while preserving the foliation $\mathfrak{F}$. Yet, the selection of these families of diffeomorphisms is not unique. As we will see, a suitable choice of $(f_x)_{x \, \in \, M}$, complemented with an additional condition on the potential $\varphi$, allow us to guarantee, after fixing a basic piece $\Lambda \times \mathbb{T}^2$, the existence of a set $B \subset \Lambda$ as requested in Theorem~\ref{teo:maintheorem-2}.


Let us exemplify this reasoning with $F_{AS}$. One way to ensure the previous property is to control the sojourns of the orbits by $\sigma$ inside the small open set $\mathcal{U} \subset M$ where the deformation of $\sigma \times L$ into $F_{AS}$ is performed. For those points $x \in \mathcal{U}$, the diffeomorphism $f_x$ is a derived from Anosov and the foliation $\mathfrak{F}$ in the fiber $\{x\} \times \mathbb{T}^2$ may be expanding by $F_{AS}$; this is precisely the behavior which, according to Theorem~\ref{teo:maintheorem-2}, we should reduce to a minimum. So, given a basic set $\Lambda$ of $\Omega(\sigma)$, a H\"{o}lder potential $\varphi$ and the equilibrium state $\nu_\varphi$ with respect to $\sigma \times L$, we will request that $F_{AS}$ and $\nu_\varphi$ satisfy the condition
\begin{equation}\label{eq:mostly-contracting}
\int_\Lambda \, \log \, \zeta(x)\,d{(\pi_1)}_*(\nu_\varphi)(x) > 0
\end{equation}
where
$$\zeta(x) = \inf_{y \,\, \in \,\,\mathbb{T}^2}\,\|{D_y f^{-1}_x}_{|_{E^{\mathrm s}_{(x,y)}}}\|$$
and $\pi_1 \colon \Lambda \times \mathbb{T}^2 \, \to \, \Lambda$ is the natural projection on the first factor. Notice that $\zeta(x)$ measures the least rate of expansion of $D_y f^{-1}_x$ along the leaf $\mathfrak{F}_{(x,y)}$ at the fiber $\{x\} \times \mathbb{T}^2$; and that condition \eqref{eq:mostly-contracting} says that, in average, these minimum rates are bigger than $1$ (hence expanding). We are left to prove that this assumption is enough to ensure that the set
$$B = \Big\{x \in \Lambda \colon \,\chi(x,y)> 0 \quad \forall \, y \, \in \, \mathbb{T}^2\Big\}$$
satisfies $B \times \mathbb{T}^2 \subset \mathcal{E}$ and $\nu_\varphi(B \times \mathbb{T}^2)=1$. This way, by Theorem~\ref{teo:maintheorem-2}, $\mu_{\varphi \circ H}$ is the unique equilibrium state for $(F_{AS}, \,\varphi \circ H)$. Now, observe that the case $\varphi \equiv 0$ is considerably simpler. Indeed, for this potential, $(\pi_1)_*(\nu_0)$ is the measure of maximal entropy of $\sigma$ (cf. \cite[Theorem 1]{NY1983}), which, besides being uniformly distributed, provides an explicit connection between the size of $\mathcal{U}$ and its $(\pi_1)_*(\nu_0)$ measure. This enables us to suitably choose, once and for all, the family of diffeomorphisms $(f_x)_{x \,\in \,\mathbb{S}^2}$ in order to master the magnitude of the deformation that builds $F_{AS}$ and, at the same time, to ascertain the validity of the assumption \eqref{eq:mostly-contracting} for $\varphi \equiv 0$ and $\nu_0$. Thus, the set
$$\mathbb{P}(\Lambda \times \mathbb{T}^2) = \left\{\varphi \in C^0(\Lambda \times \mathbb{T}^2) \colon \, \varphi \text{ is H\"{o}lder and }\,\int_\Lambda \, \log \, \zeta(x)\,d{(\pi_1)}_*(\nu_\varphi)(x) > 0\right\}$$
in non-empty (in fact, it contains every constant potential $\varphi \,\equiv \,c$, since $P_{\text{top}}(f, c) = h_{\text{top}}(f) + c$ and therefore the equilibrium states $\nu_c$ and $\nu_0$ coincide). Moreover, we will show in Section~\ref{se:potentials} that $\mathbb{P}(\Lambda \times \mathbb{T}^2)$ is an open domain inside the subset of H\"{o}lder elements of $C^0(\Lambda \times \mathbb{T}^2)$. The next result states that the class $\mathbb{P}(\Lambda \times \mathbb{T}^2)$ is an adequate choice of potentials for $F_{AS}$.

\begin{maincorollary}\label{cor:maincorollary-1}
Let $F_{AS}:\mathbb{S}^2 \times \mathbb{T}^2\to \mathbb{S}^2 \times \mathbb{T}^2$ be Abraham-Smale's example, $H$ the semi-conjugacy between $F_{AS}$ and $\sigma \times L$ given by Lemma~\ref{le:semi-conjugacy} and $\Lambda$ be a basic set of $\sigma$. Take a potential $\varphi \in \mathbb{P}(\Lambda \times \mathbb{T}^2)$ and the unique equilibrium state $\nu_{\varphi}$ associated to the $(\sigma \times L)_{|_{\Lambda \times \mathbb{T}^2}}$ and the potential $\varphi$. Then:
\begin{enumerate}
\item There exists a $\sigma$-invariant set $B \subset \Lambda$ such that $B \times \mathbb{T}^2\subset \mathcal{E}$ and $\nu_{\varphi}(B \times \mathbb{T}^2)=1$.
\medskip
\item $F_{AS}$ has a unique equilibrium state $\mu_{\varphi \circ H}$ for the potential $\varphi \circ H$.
\medskip
\item $(F_{AS},\, \mu_{\varphi\circ H})$ and $(\sigma \times L, \, \nu_\varphi)$ are measure theoretically isomorphic. In particular, $\mu_{\varphi\circ H}$ is Bernoulli.
\end{enumerate}
An analogous statement is true for Shub's example.
\end{maincorollary}


We now proceed studying $C^1$ diffeomorphisms $G \colon \mathbb{S}^2 \times \mathbb{T}^2 \, \to \, \mathbb{S}^2 \times \mathbb{T}^2$ in a small open neighborhood $\mathcal{V}$ of $F_{AS}$ in $\text{Diff}^1(\mathbb{S}^2 \times \mathbb{T}^2)$. A similar discussion may be pursued with $F_{Sb}$. We may select $\mathcal{V}$ so that, for each $G \in \mathcal{V}$, the following properties remain valid:
\begin{itemize}
\item[(a)] There exists a partially hyperbolic splitting on $\Omega(G)$ with integrable subbundles $E^*(G)$, where $* = \mathrm{s, c, u, cs}$; besides, $\text{dim } E^{\mathrm{c}}(G)=1$ and the center foliation is normally hyperbolic (cf. \cite{HPS1977}).
\medskip
\item[(b)] $G$ may not be a skew-product, but there exists a homeomorphism $\tau_G \colon \,\Omega (G) \, \to \,\Omega (F_{AS})$ such that $\tau_G \circ G \circ \tau_G^{-1}$ is a bundle map covering $\sigma$ and $C^0$-close to $G$ and $F_{AS}$ (cf. \cite[\S8]{HPS1977}). Moreover, the skew product $\tau_G \circ G \circ \tau_G^{-1}$ satisfies the conditions (1)-(3) of Section~\ref{se:setting}. Denote by  $h_G$ the corresponding semi-conjugacy provided by Lemma~\ref{le:semi-conjugacy}.
    \medskip
\item[(c)] $\Omega(G) = \Omega_1(G) \cup \Omega_2(G) \cup \Omega_3(G)$, where $\Omega_i(G) = \tau_G^{-1}(\Omega_i)$, for $i=1,2,3$, and $\Omega_i(G)$ is a transitive, locally maximal and partially hyperbolic set.
\medskip
\item[(d)] The restriction of $G$ to each $\Omega_i(G)$ is entropy-expansive.
\medskip
\end{itemize}
Therefore, the existence of equilibrium states for each restriction $G_{|_{\Omega_i(G)}}$ and every continuous potential is guaranteed. The next result provides uniqueness of these equilibrium states for a certain class of potentials.
\newpage
\begin{maintheorem}\label{teo:maintheorem-3}
Consider $G \in \mathcal{V}$, a basic set $\Lambda$ of $\,\Omega(\sigma)$ and a potential $\varphi \in \mathbb{P}(\Lambda \times \mathbb{T}^2)$. Then there is a unique equilibrium state $\mu_{\varphi\,\circ\,h_G\,\circ\,\tau_G}$ for the restriction of $G$ to the piece $\tau_G^{-1}(\Lambda \times \mathbb{T}^2)$ of $\Omega(G)$ and the continuous potential $\varphi \circ h_G \circ \tau_G$. An analogous statement applies to Shub's example.
\end{maintheorem}

Having established the existence and uniqueness of equilibrium states for H\"{o}lder potentials (up to the semi-conjugacies with $\sigma \times L$) for every diffeomorphism in a $C^1$ neighborhood of $F_{AS}$ (respectively, $F_{Sh}$), a natural question is how these equilibrium states vary with the underlying dynamics and the potential. Given a basic set $\Lambda$ of $\Omega(\sigma)$ and a potential $\varphi \in \mathbb{P}(\Lambda \times \mathbb{T}^2)$, we say that the equilibrium state $\mu_{\varphi \circ H}$ of $F_{AS}$ is \emph{statistically stable} if, given a sequence $(G_n)_{n \, \in \, \mathbb{N}}$ of diffeomorphisms in $\mathcal{V}$ converging in the $C^1$-topology to $F_{AS}$, a sequence $(\varphi_n)_{n \, \in \, \mathbb{N}}$ of maps in $\mathbb{P}(\Lambda \times \mathbb{T}^2)$ converging in the $C^0$-norm to $\varphi$ and the sequence of equilibrium states $\mu_n$ for the restriction of $ G_n$ to $\tau_{G_n}^{-1}(\Lambda \times \mathbb{T}^2)$ and $\varphi_n \circ h_{G_n} \circ \tau_{G_n}$, then any accumulation point of the sequence $(\mu_n)_{n \,\in \,\mathbb{N}}$ in the weak$^{*}$-topology 
is the equilibrium state $\mu_{\varphi \circ H}$ of $(F_{AS}, \,\varphi \circ H)$. We will prove the statistical stability of the equilibrium states of $F_{AS}$ obtained in Corollary~\ref{cor:maincorollary-1}. 

\begin{maincorollary}\label{cor:maincorollary-2} Given a basic set $\Lambda$ of $\Omega(\sigma)$ and $\varphi\in \mathbb{P}(\Lambda \times \mathbb{T}^2)$, then the map
$$(G, \, \psi) \,\in \,\mathcal{V} \times \mathbb{P}(\Lambda \times \mathbb{T}^2) \quad  \mapsto  \quad  \mu_{\psi\,\circ\,h_G\,\circ\,\tau_G} \, \, \in \,\mathscr{P}(\mathbb{S}^2 \times \mathbb{T}^2)$$
is continuous at $(F_{AS}, \, \varphi)$. A similar result holds for $F_{Sh}$.
\end{maincorollary}

\section{Proof of Theorem~\ref{teo:maintheorem-1}}\label{se:proof-Th-A}

Let $H$ be the semi-conjugacy given by Lemma~\ref{le:semi-conjugacy}. The inequality
$$P_{\text{top}}(F, \,\varphi \circ H) \,\geq \, P_{\text{top}}(\sigma \times L, \, \varphi), \quad \quad \forall \,\,\varphi \in C^0(M \times \mathbb{T}^2)$$
follows directly from \cite[Theorem 9.8]{W1981}. Regarding the reverse inequality, we start showing another connection between the pressure of two semi-conjugate dynamics.

\begin{proposition}\label{pro:dianublado} Let $T:X \to X$ and $S:Y \to Y$ be two continuous transformations on compact metric spaces $X$ and $Y$. Suppose there exists a surjective continuous map $\pi: X \to Y$ satisfying $\pi \circ T = S \circ \pi$. Then, for every potential $\varphi \in C^0(Y)$ we have
$$P_{\mathrm{top}}(T, \,\varphi \circ \pi ) \le P_{\mathrm{top}}(S,\,\varphi) + \sup_{\nu \,\in \,\mathscr{P}(Y,S)}\int\,
h_{\mathrm{top}}\big(T, \pi^{-1}(y)\big)\,d\nu(y).$$
\end{proposition}

\begin{proof} For each $\nu \in \mathscr{P}(Y,S)$, consider the set $\mathscr{P}_\nu(X,T) \subset \mathscr{P}(X,T)$ defined by
$$ \mathscr{P}_\nu(X,T)= \Big\{\mu \in \mathscr{P}(X,T) \colon \, \pi_*\mu=\nu \Big\}.$$
Observe that
\begin{equation}\label{eq:union}
\mathscr{P}(X,T)=\displaystyle\bigcup_{\nu \,\in \,\mathscr{P}(Y,S)}\,\,\mathscr{P}_\nu(X,T).
\end{equation}
Moreover, Ledrappier-Walters' formula (\cite{LW1977}) establishes that, for every $\nu \in \mathscr{P}(Y,S)$,
\begin{equation}\label{eq:LW-formula}
\sup_{\substack{\mu \,\in \,\mathscr{P}_\nu(X,T)}}\,h_\mu(T) = h_\nu(S) + \int\,h_{\mathrm{top}}\big(T, \,\pi^{-1}(y)\big)\,d\nu(y).
\end{equation}
Adding $\int\,\varphi\,d\nu$ to both sides of the equality \eqref{eq:LW-formula} we obtain
\begin{equation} \label{eq:LW-formula-1}
\sup_{\substack{\mu \,\in \,\mathscr{P}_\nu(X,T)}}\,\Big\{h_\mu(T)+ \int\varphi\circ \pi\,d\mu \Big\}= h_\nu(S) +\int\,\varphi\,d\nu+\int\,h_{\mathrm{top}}\big(T, \,\pi^{-1}(y)\big)\,d\nu(y).
\end{equation}
Taking in \eqref{eq:LW-formula-1} the supremum over all the measures $\nu \in \mathscr{P}(Y,S)$ and applying the Variational Principle to $S: Y \to Y$ and $\varphi$, we get
\begin{eqnarray*}
&&\sup_{\substack{\nu \,\in \,\mathscr{P}(Y,S)}}\, \left\{\sup_{\substack{\mu \,\in \,\mathscr{P}_\nu(X,T)}}\,\Big\{h_\mu(T) + \int\varphi\circ \pi\,d\mu \Big\} \right\}  \\
&\quad& \quad \leq P_{\mathrm{top}}(S,\varphi) \,+ \sup_{\substack{\mu \,\in \,\mathscr{P}_\nu(X,T)}} \int\,h_{\mathrm{top}}\big(T, \,\pi^{-1}(y)\big)\,d\nu(y).
\end{eqnarray*}
On the other hand, it follows from~\eqref{eq:union}  that
$$\sup_{\substack{\mu \,\in \,\mathscr{P}(X,T)}}\,\Big\{h_\mu(T)+ \int\varphi\circ \pi\,d\mu \Big\}
= \sup_{\substack{\nu \,\in \,\mathscr{P}(Y,S)}}\, \left\{\sup_{\substack{\mu \,\in \,\mathscr{P}_\nu(X,T)}}\,\Big\{h_\mu(T)+ \int\varphi\circ \pi\,d\mu \Big\} \right\}$$
and so, again by the Variational Principle, we conclude that
\begin{equation*}
P_{\mathrm{top}}(T,\varphi\circ \pi) \le  P_{\mathrm{top}}(S,\varphi)+\sup_{\substack{\mu \,\in \,\mathscr{P}_\nu(X,T)}}\int\,h_{\mathrm{top}}\big(T, \,\pi^{-1}(y)\big)\,d\nu(y)
\end{equation*}
thus completing the proof of the proposition.
\end{proof}

Now, for the skew-products introduced in Section~\ref{se:setting}, as the leaves of the lamination $\mathfrak{F}$ are one-dimensional, each fiber of the semi-conjugacy carries no entropy (this is not necessarily true if the dimension of those leaves is larger than one; we refer the reader to \cite[Remark at page 628]{NY1983}). More precisely:

\begin{lemma}\cite[Lemma 3]{NY1983}\label{le:zero-relative-entropy} For every $(x,y) \in M \times \mathbb{T}^2$, one has $h_{\mathrm{top}}\big(F, H^{-1}(x,y)\big)=0$.
\end{lemma}

Consequently, in this setting, Proposition~\ref{pro:dianublado} may be rewritten as
$$P_{\mathrm{top}}(F, \,\varphi \circ H) \le P_{\mathrm{top}}(\sigma\times L,\,\varphi), \quad \quad \forall \,\,\varphi \in C^0(M \times \mathbb{T}^2).$$
This ends the proof of part (1) of Theorem~\ref{teo:maintheorem-1}.\\

Concerning statement (2) of Theorem~\ref{teo:maintheorem-1}, let $\mu_{\varphi\circ H}$ be an equilibrium state for $(F, \,\varphi \circ H)$. We need to show that $\nu_\varphi=H_*(\mu_{\varphi\circ H})$, which belongs to $\mathscr{P}(M \times \mathbb{T}^2, \sigma\times L)$, is
an equilibrium state for $(\sigma\times L, \,\varphi)$. By Lemma~\ref{le:zero-relative-entropy}, equation~\eqref{eq:LW-formula-1} and part (1) of Theorem \ref{teo:maintheorem-1}, we deduce that
\begin{eqnarray*}
\begin{split}
h_{\nu_\varphi}(\sigma\times L) + \int\varphi\,d\,\nu_\varphi
&= \sup_{\substack{\mu \,\in \,\mathscr{P}_{\nu_\varphi}(M \times \mathbb{T}^2,\,F)}} \,\Big\{h_\mu(F) + \int
\varphi\circ H \,d\mu \Big\}
\\
&\ge h_{\mu_{\varphi \circ H}}(F) + \int
\varphi \circ H \,d\mu_{\varphi \circ H} \\
& = P_{\text{top}}(F, \varphi\circ H)
 \\
& = P_{\text{top}}(\sigma \times L, \, \varphi).
\end{split}
\end{eqnarray*}
Thus, $\nu_\varphi$ is an equilibrium state of $\sigma \times L$ and $\varphi$. This ends the proof of Theorem~\ref{teo:maintheorem-1} (2).

\medskip

Let us now show item (3) of Theorem~\ref{teo:maintheorem-1}. Assume that $\nu_\varphi$ is an equilibrium state of $\sigma \times L$ and $\varphi$, and that the set $\mathcal{A}$ of points in $M \times \mathbb{T}^2$ where $H$ is injective has full $\nu_\varphi$ measure.

\begin{lemma} $\mathcal{A}$ is $(\sigma \times L)$-invariant.
\end{lemma}

\begin{proof} As $F$ and $\sigma \times L$ are bijections and $H \circ F = (\sigma \times L) \circ H$, we have $H^{-1} \circ (\sigma \times L)^{-1}(\cdot) \,\supseteq \,F^{-1} \circ H^{-1}(\cdot)$. Therefore, if $z \in \mathcal{A}$, then
$$H^{-1}(z) = H^{-1} \circ (\sigma \times L)^{-1}((\sigma \times L)(z)) \,\supseteq \,F^{-1} \circ H^{-1}((\sigma \times L)(z))$$
and so, as $H^{-1}(z)$ is a singular set and $F$ is injective, $H^{-1}((\sigma \times L)(z))$ must be singular as well. Thus, $(\sigma \times L)(\mathcal{A}) \subset \mathcal{A}$. A similar argument proves the reverse inclusion.
\end{proof}

Thus, $H^{-1}(\mathcal{A})$ is $F$-invariant, 
and, since $H_{|\mathcal{A}}$ is injective, we conclude that:

\begin{lemma} $H_{|\mathcal{A}}$ is a topological conjugacy between $F_{|_{H^{-1}(\mathcal{A})}}$ and $(\sigma \times L)_{|_{\mathcal{A}}}$.
\end{lemma}

\begin{proof}
The main difficulty in proving this statement is the fact that, having in mind possible applications (as happens with the subset $B \times \mathbb{T}^2$ in Theorem~\ref{teo:maintheorem-2}), we are not assuming that the set $\mathcal{A}$
is compact. Thus, it is not immediate that the injectivity of $H$ at $\mathcal{A}$ yields the continuity of $H^{-1}$ in $H(\mathcal{A})$. Consider a sequence $\Big((x_n,\,y_n)\Big)_{n\, \in \mathbb{N}}$ of points in $\mathcal{A}$ such that the sequence $(H(x_n, \,y_n))_{n\, \in \mathbb{N}}$ converges to $H(x_0, \,y_0)$ for some $(x_0, \,y_0) \in \mathcal{A}$. We need to show that $\lim_{n \, \to \, +\infty}\ (x_n,\,y_n) = (x_0, \,y_0).$ As $M \times \mathbb{T}^2$ is compact, the sequence $\Big((x_n,\,y_n)\Big)_{n\, \in \mathbb{N}}$ has accumulation points in $M \times \mathbb{T}^2$. Suppose that $\beta_1$ and $\beta_2$ are accumulated by the sequence $\Big((x_n,\,y_n)\Big)_{n\, \in \mathbb{N}}$. Then, by the continuity of $H$, we get
$$\lim_{n \, \to \, +\infty}\ H(x_n,\,y_n) = H(x_0,\,y_0) = H(\beta_1) = H(\beta_2).$$
Therefore, as $H$ is injective at $(x_0,\,y_0)$, we infer that $\beta_1 = \beta_2 = (x_0,\,y_0)$. Hence, $\Big((x_n,\,y_n)\Big)_{n\, \in \mathbb{N}}$ converges to $(x_0,\,y_0)$, and so $H^{-1}$ is continuous in $H(\mathcal{A})$.
\end{proof}

Consequently, the map $(H^{-1}_{|_{\mathcal{A}}})_*$ is an isomorphism between the domains and $\mathscr{P}(\mathcal{A},(\sigma \times L)_{|_{\mathcal{A}}})$ and $\mathscr{P}(H^{-1}(\mathcal{A}),F_{|_{H^{-1}(\mathcal{A})}})$. Thus, $(H^{-1})_*(\nu_\varphi)$ is an equilibrium state for $\varphi \circ H$ with respect to $F$.

\section{Proof of Theorem~\ref{teo:maintheorem-2}}\label{se:proof-Th-B}

The key idea to prove Theorem~\ref{teo:maintheorem-2} is Lemma 4 of \cite{NY1983} which ascertains that, under the assumption that $F^{-m}$ expands lengths in the central leaves $\mathfrak{F}$ for large enough values of $m \in \mathbb{N}$ (a property that is conveyed by the set $B \times \mathbb{T}^2$), then the restriction of the semi-conjugacy $H_{|_{B \times \mathbb{T}^2}}$ becomes an almost conjugacy.

\begin{lemma}\cite[Lemma 4]{NY1983} Suppose there is a $\sigma$-invariant set $B\subset \Lambda$ such that $\chi(x,y) > 0$ for every $(x,y) \in B \times \mathbb{T}^2$. Then the restriction of $H$ to the set $B \times \mathbb{T}^2$ is a homeomorphism.
\end{lemma}

Notice that this lemma also indicates that $B \times \mathbb{T}^2\subset \mathcal{A}$.  
So, as by assumption $\nu_{\varphi}(B\times \times \mathbb{T}^2)=1$, we have $\nu_{\varphi}(\mathcal{A})=1$ and Theorem~\ref{teo:maintheorem-1} (3) applies.

Assume that $\sigma \times L$ has a unique equilibrium state $\nu_\varphi$ with respect to $\sigma \times L$ and a H\"{o}lder potential $\varphi$. We need to show that $F$ also has only one equilibrium state for $\varphi \circ H$. Suppose, otherwise, that $F$ has two equilibrium states $\mu_1$ and $\mu_2$ for $\varphi \circ H$. Then Theorem~\ref{teo:maintheorem-1} (2) indicates that $H_*(\mu_1) = H_*(\mu_2)= \nu_\varphi.$ Moreover, by assumption, there exists a $\sigma$-invariant set $B$ such that
$$B \times \mathbb{T}^2 \subset \mathcal{E}=\Big\{(x,y) \in M \times \mathbb{T}^2 \colon \,\chi(x,y)> 0\Big\}$$
and $\nu_\varphi(B \times \mathbb{T}^2) = 1$. Thus, as $H$ is surjective and is the identity on the first coordinate,
\begin{equation}\label{eq:miu=niu}
\mu_1(B \times \mathbb{T}^2) = \mu_1(H^{-1}(B \times \mathbb{T}^2)) = \nu_\varphi(B \times \mathbb{T}^2) = 1
\end{equation}
and similarly $\mu_2(B \times \mathbb{T}^2) = 1.$ Besides, since $H_{|B \times \mathbb{T}^2}$ is a topological conjugacy between $F_{|_{B \times \mathbb{T}^2}}$ and $(\sigma \times L)_{|_{B \times \mathbb{T}^2}}$, the map $H_*$ is an isomorphism between the domains $\mathscr{P}(B \times \mathbb{T}^2,F_{|_{B \times \mathbb{T}^2}})$ and $\mathscr{P}(B \times \mathbb{T}^2,(\sigma \times L)_{|_{B \times \mathbb{T}^2}})$. Thus, $\mu_1 = (H^{-1})_*(\nu_\varphi) = \mu_2$.

\medskip

\subsection{A generalization} Theorems~\ref{teo:maintheorem-1} and \ref{teo:maintheorem-2} together admit a more general statement whose proof is a straightforward adaptation of the arguments presented in this and the previous section.

Let $f \colon X \,\to\,X$ be an expansive homeomorphism of a compact metric space $X$ with the specification property. Consider a continuous extension $F$ of $f$ through a continuous surjective map $H$, and take a H\"{o}lder potential $\varphi$ and the unique (ergodic) equilibrium state $\nu_\varphi$ of $f$ and $\varphi$. Suppose that the two following conditions are fulfilled:
\medskip
\begin{enumerate}
\item $h_{\text{top}}(F, \,H^{-1}(x))=0$ for every $x \in X$.
\medskip
\item $\nu_\varphi\left(\Big\{H(x) \in X \colon \, H^{-1}(H(x)) = \{x\}\Big\}\right) = 1$.
\end{enumerate}
Then the measure ${(H^{-1})}_*(\nu_\varphi)$ is $F$-invariant, ergodic, and the unique equilibrium state of $F$ and $\varphi \circ H$. This result may be compared with \cite[Theorem 1.5]{BFSV2012}.

\section{Proof of Corollary~\ref{cor:maincorollary-1}}\label{se:proof-Cor-C}

We start recalling the construction of Abraham-Smale and Shub's examples. Let $\sigma \colon  M \to M$ be a diffeomorphism satisfying the Spectral Decomposition Theorem. The tangent space of $M$ admits a hyperbolic splitting $T\, M = E^{\mathrm{ss}} \oplus E^{\mathrm{uu}}$ such that, for some uniform constant $0 < \lambda < 1$,
$$\max\,\left\{\|D\sigma|_{E^{\mathrm{ss}}}\|, \,\,\|D\sigma^{-1}|_{E^{\mathrm{uu}}}\|\right\} < \lambda.$$
Assume that $\sigma$ has two fixed points $p$ and $q$ homoclinically related (that is, both transversal intersections $W^{\mathrm s}(p)\pitchfork W^{\mathrm u}(q)$ and $W^{\mathrm u}(p)\pitchfork W^{\mathrm s}(q)$ are non-empty). Afterwards, take a smooth family of torus diffeomorphisms $f_x \colon \,\mathbb{T}^2 \,\to\, \mathbb{T}^2$ indexed by $x \in M$ and satisfying the following properties:
\medskip
\begin{itemize}
\item[(a)] $T\,\mathbb{T}^2 = E^{\mathrm{c}}(f_x) \oplus E^{\mathrm{u}}(f_x)$, a splitting invariant under $Df_x$ and for which there exist constants $0 < \gamma_1 < \gamma_2 < 1$ such that
$$\|D{f_x^{-1}}|_{E^{\mathrm{u}}(f_x)}\|  <  \gamma_1  \quad \quad \text{and} \quad \quad  \gamma_2 < \|D f_x|_{E^{\mathrm{c}}(f_x)} \| \leq \,\gamma_1^{-1}.$$
We may assume, taking a power of $f_x$ if necessary, that $\lambda < \gamma_1$.
\medskip
\item[(b)] For every $x \in M$, the diffeomorphism $f_x$ preserves cone fields $\mathcal{C}^\mathrm{cs}$ and $\mathcal{C}^\mathrm{u}$.
\medskip
\item[(c)] The map $f_p$ is Anosov, while $f_q$ is a derived from Anosov.
\medskip
\item[(d)] There is $\theta \in \mathbb{T}^2$ such that $f_x(\theta) = \theta$ for every $x$, and $\theta$ is a saddle of $f_p$ and a source for $f_q$.
\medskip
\end{itemize}

Now, consider the skew-product
$$F\colon \,M \times \mathbb{T}^2 \,\to\, M \times \mathbb{T}^2,\qquad F(x,y)=\left(\sigma(x),f_x(y)\right).$$
In \cite{HPS1977} it was shown that $F$ can be obtained as a deformation of $\sigma \times L$, where $L$ is a linear Anosov diffeomorphism of the two-torus, satisfying the conditions (1)-(3) made explicit in Section~\ref{se:setting}. Moreover, the resulting $F$ is partially hyperbolic, with central foliation $\mathfrak{F}$ whose leaves are tangent to the $E^{\mathrm c}$-direction (cf. \cite{PS2006}).

The example of Abraham-Smale $F_{AS}$ is obtained by the previous reasoning considering $M$ to be equal to the two-sphere $\mathbb{S}^2$ with the dynamics $\sigma$ of Smale's horseshoe. The family $(f_x)_{x \,\in \, M}$ is the result of a homotopic deformation of the linear Anosov diffeomorphism $L$ induced by the matrix
$A=\tiny
\left(\begin{array}{cc}
2 & 1\\
1 & 1
\end{array}
\right)$
\normalsize
at $p$ until the deformation reaches a derived from Anosov diffeomorphism when $x$ attains $q$. In the case of Shub's example, the diffeomphism $F_{Sh}$ is obtained taking $M = \mathbb{T}^2$ and $\sigma$ an Anosov diffeomorphism having two fixed points. Both skew-products $F_{AS}$ and $F_{Sh}$ are $C^0$-perturbations of $\sigma \times L$ supported on a small ball with radius $\rho>0$ centered at the fixed point $q$ of $\sigma$ in $M$. Outside that ball, the maps coincide with $\sigma \times L$, hence we may carry out the above construction so that the $C^0$-distance between $F_{AS}$ (respectively, $F_{Sh}$) and $\sigma \times L$ is of the order of $\rho$.

Let us be more accurate, exemplifying with the case of $F_{AS}$. The matrix $A$ is diagonalizable, with eigenvalues $0 <\lambda_s <1$ and $\lambda_u > 1$ such that $\lambda_s \,\lambda_u = 1$. Denote by $B_r(x)$ the $r$-ball centered at $x \in \mathbb{S}^2$. Then (see Figure \ref{fig:H}):
\medskip
\begin{itemize}
\item[(a)] For $x \,\in \,\mathbb{S}^2 \setminus B_\rho(q)$, the diffeomorphism $f_x$ is equal to $L$.
\medskip
\item[(b)] For $x \,\in \,B_{\rho/2}(q)$, the diffeomorphism $f_x$ is equal to a derived from the Anosov $L$, whose (lower triangular) linear part is
\tiny
$\left(\begin{array}{cc}
\lambda_u & 0\\
* & \lambda_c
\end{array}
\right)$
\normalsize
with $1 < \lambda_c < \lambda_u$.
\medskip
\item[(c)] For $x \,\in \,B_\rho(q) \setminus B_{\rho/2}(q)$, the diffeomorphism $f_x$ is a transition between $L$ and the previous derived from Anosov, with linear part
\tiny
$\left(\begin{array}{cc}
\lambda_u & 0\\
* & \gamma(x)
\end{array}
\right)$
\normalsize
where $\gamma(x) > 0$.
\end{itemize}
\begin{figure}
\centering
\begin{overpic}[scale=.1,
  ]{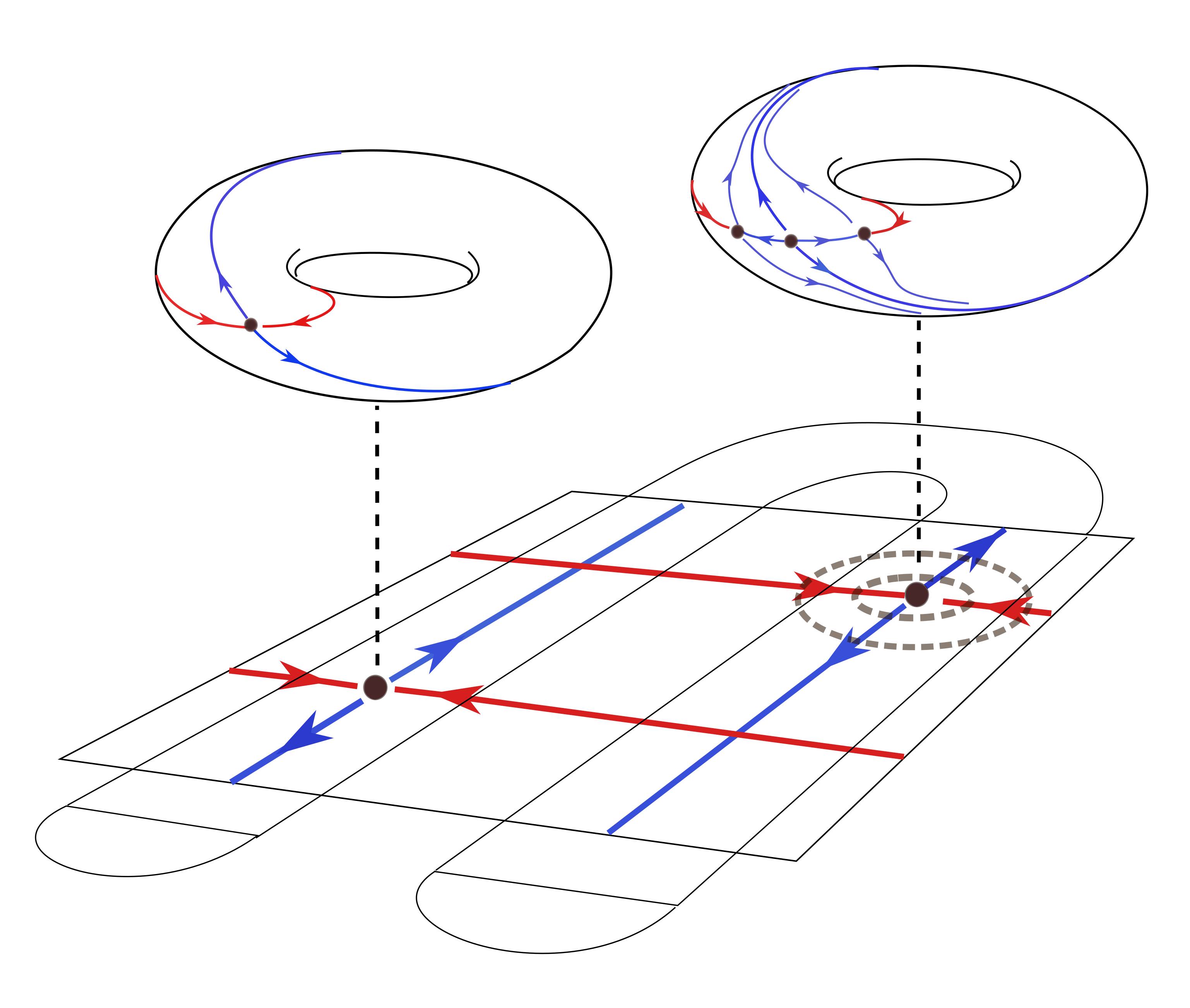}
            \put(25,75){Anosov}
           \put(21.5,59){\small{$\theta$}}
            \put(29,21){{$p$}}
     \put(76,27){{$q$}}
           \put(66.5,65.5){\small{$\theta$}}
            \put(60,82){{Derived from Anosov}}
    \end{overpic}
\caption{Homotopic deformation from an Anosov diffeomorphism to a Derived from Anosov.}
\label{fig:H}
\end{figure}
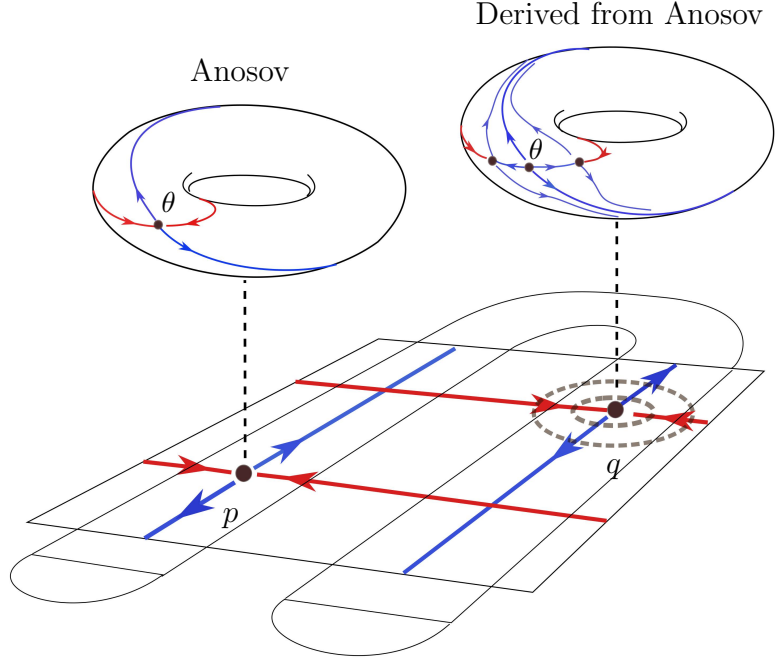
Observe that we have control on the size $\rho$ of the neighborhood of $q$ where the perturbation occurs, on the value of $\lambda_c$ and on the map $x \in B_\rho(q) \setminus B_{\rho/2}(q) \, \mapsto \, \gamma(x)$. Selecting them appropriately, we are able to bind to this construction the condition \eqref{eq:mostly-contracting} for $\varphi \equiv 0$ and the measure of maximal entropy $\nu_0$ of $\sigma \times L$. Indeed, the map $\zeta$ is given by
$$\zeta(x)=\left\{\begin{array}{ll}
\lambda_s^{-1} & \quad \quad \text{if $\,x \, \in \,\mathbb{S}^2 \setminus B_\rho(q)$} \smallskip\\
\lambda_c^{-1} & \quad \quad \text{if $\,x \, \in \, B_{\rho/2}(q)$}  \smallskip\\
\gamma(x)^{-1} & \quad \quad  \text{if $\,x \, \in \, B_\rho(q) \setminus B_{\rho/2}(q)$}.
\end{array}
\right.$$
So, we construct $F_{AS}$ requesting that
\medskip
\begin{itemize}
\item[(H1)] $\rho$ is small enough so that $(\pi_1)_*(\nu_0)(B_\rho(q)) < \frac{\log \, \lambda_u}{\log \,\lambda_u \,+ \,\log \,\lambda_c}$;
\medskip
\item[(H2)] $\gamma(x) \leq \lambda_c$ for every $x \in B_\rho(q) \setminus B_{\rho/2}(q)$.
\medskip
\end{itemize}
This way, we obtain
$$
\int_\Lambda \, \log \, \zeta(x)\,d{(\pi_1)}_*(\nu_0)(x) \geq \log\,(\lambda_u) - [\log \, \lambda_u + \log \,\lambda_c](\pi_1)_*(\nu_0)(B_\rho(q)) > 0.
$$

Having finished the construction of $F_{AS}$, take a basic set $\Lambda$ of $\Omega(\sigma)$, a potential $\varphi \in \mathbb{P}(\Lambda \times \mathbb{T}^2)$ and the equilibrium state $\nu_\varphi$ with respect to $\sigma \times L$ and $\varphi$. We will show that the set
$$B = \Big\{x \in \Lambda \colon \,\chi(x,y)> 0 \quad \forall \, y \, \in \, \mathbb{T}^2\Big\}$$
satisfies $B \times \mathbb{T}^2 \subset \mathcal{E}$ and $\nu_\varphi(B \times \mathbb{T}^2)=1$. Accordingly, by Theorem~\ref{teo:maintheorem-2}, $\mu_\varphi = (H_{|_{B \times \mathbb{T}^2}}^{-1})_*(\nu_\varphi)$ is the unique equilibrium state for $F_{AS}$ and $\varphi \circ H$.

We begin noticing that any ergodic $(\sigma \times L)$-invariant Borel probability measure projects via $\pi_1$ to an ergodic $\sigma$-invariant probability measure.
Besides, as $\sigma \times L$ is a $C^1$ Axiom A diffeomorphism and $\varphi$ is H\"{o}lder, there is a unique equilibrium state of $\varphi$ with respect to $\sigma \times L$, which is a Bernoulli probability measure \cite{Bo1975} since $\sigma$ is topologically mixing for both $\Lambda=\text{ Horseshoe}$ and $\Lambda=\mathbb{T}^2$. By Theorem~\ref{teo:maintheorem-1}, such a unique equilibrium state must be $\nu_\varphi$. Taking into account that $\mathfrak{F}$ is a one-dimensional foliation, the Ergodic Theorem of Birkhoff informs that, for every $y \in \mathbb{T}^2$ and ${(\pi_1)}_*(\nu_\varphi)$ almost every $x \in \Lambda$, one has
\begin{eqnarray*}
\chi(x,y) &=& \lim_{n \, \to \, + \infty} \, \frac{1}{n} \, \log \|{D_{(x,y)} \,F^{-n}}_{|_{E^s_{(x,y)}}}\| = \int_\Lambda \, \log \, {\|D_y f^{-1}_x}_{|_{E^s_{(x,y)}}}\|\,d{(\pi_1)}_*(\nu_\varphi) \\
&\geq & \int_\Lambda \, \log \, \zeta(x)\,d{(\pi_1)}_*(\nu_\varphi)(x) > 0.
\end{eqnarray*}
Therefore the $\sigma$-invariant set
$$B = \Big\{x \in \Lambda \colon \,\chi(x,y)> 0 \quad \forall \, y \, \in \, \mathbb{T}^2\Big\}$$
satisfies ${(\pi_1)}_*(\nu_\varphi)(B)=1.$ This means that $\nu_\varphi(B \times \mathbb{T}^2) = \nu_\varphi(\pi^{-1}_1(B)) = {(\pi_1)}_*(\nu_\varphi)(B) = 1$, which is what we needed to finish the proof of Corollary~\ref{cor:maincorollary-1}.

\section{Selection of the potentials}\label{se:potentials}

In what follows $\mathcal{H}(\Lambda \times \mathbb{T}^2) \subset C^0(\Lambda \times \mathbb{T}^2)$ stands for the subset of H\"{o}lder real-valued continuous maps with domain $\Lambda \times \mathbb{T}^2$. We will explain why the set $\mathbb{P}(\Lambda \times \mathbb{T}^2)$ of H\"{o}lder potentials which comply with the condition \eqref{eq:mostly-contracting}, with respect to either $F_{AS}$ or $F_{Sh}$, is an open subset of $\mathcal{H}(\Lambda \times \mathbb{T}^2)$.

\begin{proposition}\label{prop:stat.stab.}
Given $\varphi \in \mathcal{H}(\Lambda \times \mathbb{T}^2)$, consider a sequence $(\varphi_n)_{n \, \in \, \mathbb{N}}$ of maps in $\mathcal{H}(\Lambda \times \mathbb{T}^2)$ which converges in the $C^0$-norm to $\varphi$. Take the sequence of unique equilibrium states $\nu_n$ for $\sigma \times L$ and $\varphi_n$, and assume that the sequence $(\nu_n)_{n \,\in \,\mathbb{N}}$ converges in the weak$^{*}$-topology to a Borel probability measure $\eta$ in $\Lambda \times \mathbb{T}^2$. Then $\eta$ is the equilibrium state $\nu_\varphi$ of $\sigma \times L$ and $\varphi$.
\end{proposition}

\begin{proof} It is immediate that the probability measure $\eta$ is $\sigma \times L$-invariant.
So, we are left to show that
$$P_{\text{top}}(\sigma \times L, \varphi) = h_{\eta}(\sigma \times L) + \int\,\varphi\,d\eta.$$
As $\sigma \times L$ is expansive, its entropy map is upper semi-continuous, so
$$\limsup_{n \, \to \, +\infty}\,\,h_{\nu_n}(\sigma \times L) \leq  h_{\eta}(\sigma \times L).$$
On the other hand,
$$\lim_{n \, \to \, +\infty}\,\,\int \, \varphi_n \, d\nu_n = \int \, \varphi \, d\eta$$
since
\begin{eqnarray*}
\Big|\int \,\varphi_n \, d\nu_n - \int \,\varphi \, d\eta \Big| &\leq& \Big|\int \,\varphi_n \, d\nu_n - \int \,\varphi \, d\nu_n\Big| + \Big|\int \,\varphi \, d\nu_n - \int \,\varphi \, d\eta\Big|\\
&\leq& \|\varphi_n - \varphi\|_{C^0} + \Big|\int \,\varphi \, d\nu_n - \int \,\varphi \, d\eta\Big| \quad \stackrel{n \, \to \,+ \infty}{\longrightarrow} \quad 0.
\end{eqnarray*}

From \cite[Theorem 9.7]{W1981} and the fact that $h_{\mathrm{top}}(\sigma \times L) < + \infty$, we conclude that the pressure operator of $\sigma \times L$ is Lipschitz with respect to the potential.
Thus,
$$\lim_{n \, \to \, +\infty}\,\,P_{\text{top}}(\sigma \times L, \,\varphi_n) = P_{\text{top}}(\sigma \times L, \,\varphi)$$
and so
\begin{eqnarray*}
P_{\text{top}}(\sigma \times L, \,\varphi) &=& \lim_{n \, \to \, +\infty}\,\,P_{\text{top}}(\sigma \times L, \,\varphi_n) \\
&=& \limsup_{n \, \to \, +\infty}\,\,\Big[h_{\nu_n}(\sigma \times L) + \int\,\varphi_n \,d\nu_n\Big]\\
&=& \limsup_{n \, \to \, +\infty}\,\,h_{\nu_n}(\sigma \times L) + \lim_{n \, \to \, +\infty}\,\, \int\,\varphi_n\,d\nu_n\\
&\leq& h_{\eta}(\sigma \times L) + \int\,\varphi \,d\eta.
\end{eqnarray*}
Therefore, $P_{\text{top}}(\sigma \times L, \,\varphi) = h_{\eta}(\sigma \times L) + \int\,\varphi\,d\eta$ and $\eta$ is an equilibrium state of $\sigma \times L$ and $\varphi$. By uniqueness of this equilibrium state, $\eta = \nu_\varphi$ and the proof of the proposition is complete.
\end{proof}

Consider now the skew-product $F_{AS}$ (the argument is similar if we use $F_{Sh}$ instead) and the transformation
\begin{eqnarray*}
\mathcal{L} \colon \,\, \mathcal{H}(\Lambda \times \mathbb{T}^2) \quad &\to& \quad  \mathbb{R} \\
\varphi \quad &\mapsto& \quad \int_\Lambda \, \log \, \zeta(x)\,\,d{(\pi_1)}_*(\nu_\varphi)(x).
\end{eqnarray*}
Observe that, if $\mathcal{L}$ is continuous, then the set $\mathbb{P}(\Lambda \times \mathbb{T}^2) = \mathcal{L}^{-1}(]0, + \infty[)$ is open, as claimed.

\begin{lemma}\label{le:ergodic} The map $\mathcal{L}$ is continuous.
\end{lemma}

\begin{proof} Take a sequence $(\varphi_n)_{n \, \in \, \mathbb{N}}$ of maps in $\mathcal{H}(\Lambda \times \mathbb{T}^2)$ converging in the $C^0$-topology to $\varphi \in \mathcal{H}(\Lambda \times \mathbb{T}^2)$. Let $\nu_{\varphi_n}$ and $\nu_{\varphi}$ the unique equilibrium states of $\sigma \times L$ associated to $\varphi_n$ and $\varphi$, respectively. By Proposition~\ref{prop:stat.stab.}, the map
$$\varphi \, \in \,\, \mathcal{H}(\Lambda \times \mathbb{T}^2) \quad \mapsto \quad\, \nu_\varphi \in\mathscr{P}(\Lambda \times \mathbb{T}^2,\sigma\times L)$$
varies continuously. Moreover, as $F_{AS}$ is $C^1$ and partially hyperbolic, the function
$$ x \,\in \,\Lambda \quad \mapsto \quad \zeta(x) = \inf_{y \,\, \in \,\,\mathbb{T}^2}\,\|{D_y f^{-1}_x}_{|_{E^{\mathrm s}_{(x,y)}}}\|$$
belongs to $C^0(\Lambda)$ (see \cite{G2006}), and so the map $\log\, (\zeta \circ \pi_1)$ belongs to $C^0(\Lambda \times \mathbb{T}^2).$
Finally, we notice that, as the sequence $(\nu_{\varphi_n})_{n \,\in \,\mathbb{N}}$ converges in the weak$^{*}$-topology to $\nu_\varphi$, we get
$$\mathcal{L}(\varphi_n) = \int_\Lambda \, \log\,\zeta\,\,d{(\pi_1)}_*(\nu_{\varphi_n}) = \int_\Lambda \, \log \,(\zeta \circ \pi_1)\,d\nu_{\varphi_n}$$
and
$$\lim_{n \, \to \, +\infty}\,\int_\Lambda \, \log \, (\zeta \circ \pi_1)\,d\nu_{\varphi_n} = \int_\Lambda \, \log\,(\zeta \circ \pi_1)\,d\nu_{\varphi} = \mathcal{L}(\varphi).$$
This ends the proof of the lemma.
\end{proof}

\section{Proof of Theorem~\ref{teo:maintheorem-3}}\label{se:perturbations}

Recall that the neighborhood $\mathcal{V}$ of $F_{AS}$ has been chosen so 
that there exists a homeomorphism $\tau_G \colon \Omega (G) \, \to \,\Omega (F_{AS})$ such that $\tau_G \circ G \circ \tau_G^{-1}$ is a skew-product $C^0$-close to $G$ and $F_{AS}$, and defined by
\begin{eqnarray*}
\tau_G \circ G \circ \tau_G^{-1} \colon \quad \Omega (\sigma)\times\mathbb{T}^2 \quad &\to& \quad \Omega (\sigma)\times\mathbb{T}^2\\
(x,\,y) \quad &\mapsto& \quad \left(\sigma(x), \, g_x(y)\right).
\end{eqnarray*}
Moreover,
$\tau_G \circ G \circ \tau_G^{-1}$ is 
normally hyperbolic to the foliation $\Big\{\{x\} \times \mathbb{T}^2\Big\}_{x \, \in \, \Omega(\sigma)}$
and plaque expansive with respect to the center foliation.
We will briefly verify that $\tau_G \circ G \circ \tau_G^{-1}$
satisfies the properties (1)-(3) of Section~\ref{se:setting}. 

\begin{lemma} $\tau_G \circ G \circ \tau_G^{-1}$ is homotopic to $\sigma \times L$.
\end{lemma}

\begin{proof}
According to \cite[Remark 1, \S8]{HPS1977}, for every $x \in \mathbb{S}^2$ the fiber diffeomorphism $f_x$ is conjugate to $g_x$, and the conjugacy $\Delta_x \colon \mathbb{T}^2 \,\to \, \mathbb{T}^2$, which satisfies $\Delta_x^{-1} \circ f_x \circ \Delta_x = g_x$, is $C^0$-close and homotopic to the identity.
If $\mathcal{G} \colon \mathbb{S}^2 \times \mathbb{T}^2 \times [0,1] \, \to \, \mathbb{T}^2$ is a homotopy between $F_{AS}$ and $\sigma \times L$ with trivial fiber $\mathbb{T}^2$ and such that $\mathcal{G}(x,y,0)=f_x(y)$ and $\mathcal{G}(x,y,1)=L(y)$, then the map $\mathcal{I} \colon \mathbb{S}^2 \times \mathbb{T}^2 \times [0,1] \, \to \, \mathbb{T}^2$ defined by
$$(x,y,t) \in \mathbb{S}^2 \times \mathbb{T}^2 \times [0,1] \quad \mapsto \quad \mathcal{I}(x,y,t) = \Delta_x^{-1} \,\Big(\mathcal{G}(x, \,\Delta_x(y),\, t)\Big)$$
is a homotopy between $\tau_G \circ G \circ \tau_G^{-1}$ and $\sigma \times L$.
\end{proof}

Therefore, the skew product $\tau_G \circ G \circ \tau_G^{-1}$ satisfies the condition (2) of Section~\ref{se:setting}. As the assumptions (1) and (3) of of Section~\ref{se:setting} are robust, they hold for $\tau_G \circ G \circ \tau_G^{-1}$ as well if $G$ is close enough to $F_{AS}$. So, from Lemma~\ref{le:semi-conjugacy}, we may find a continuous surjective map $h_G \colon\, \Omega(F_{AS}) \, \to \, \Omega(\sigma \times L)$ such that $h_G \circ (\tau_G \circ G \circ \tau_G^{-1}) = (\sigma \times L) \circ h_G$. Hence, Theorem~\ref{teo:maintheorem-1} may be applied to $\tau_G \circ G \circ \tau_G^{-1}$, and so
\begin{equation}\label{eq:pressures}
P_{\text{top}}(G, \, \varphi \circ h_G \circ \tau_G) = P_{\text{top}}(\tau_G \circ G \circ \tau_G^{-1}, \, \varphi \circ h_G) = P_{\text{top}}(\sigma \times L, \, \varphi).
\end{equation}
In particular, we deduce that the topological entropy is constant in a neighborhood of $F_{AS}$.
We will now check condition \eqref{eq:mostly-contracting} for $G\in \mathcal{V}$.
Given a basic set $\Lambda$ of $\Omega(\sigma)$, the following diagrams commute
\begin{center}
$$\begin{array}{ccc}
\tau_G^{-1}(\Lambda \times  \mathbb{T}^2) \quad \quad &\stackrel{G}{\longrightarrow}& \quad \quad \tau_G^{-1}(\Lambda \times  \mathbb{T}^2) \medskip \\
\downarrow \,\tau_G \quad \quad & & \quad \quad \downarrow \, \tau_G \medskip \\
\Lambda \times  \mathbb{T}^2 \quad \quad &\stackrel{\quad \tau_G \,\circ \,G \,\circ \,\tau_G^{-1}}{\longrightarrow}& \quad \quad \Lambda \times  \mathbb{T}^2 \medskip \\
\downarrow \,h_G \quad \quad & & \quad \quad \downarrow \, h_G \medskip \\
\Lambda \times  \mathbb{T}^2  \quad \quad &\stackrel{\sigma \,\times\, L}{\longrightarrow}& \quad \quad \Lambda \times  \mathbb{T}^2
\end{array}$$
\end{center}
\medskip
and the condition \eqref{eq:mostly-contracting}, which is $C^0$-open with respect to the dynamics, is also valid for $\tau_G \circ G \circ \tau_G^{-1}$.  Indeed, as $G$ is $C^1$-near $F_{AS}$, the homeomorphism $\tau_G$ is $C^0$-near the inclusion of $\Omega(G)$ in $\mathbb{S}^2 \times \mathbb{T}^2$ and the center foliations for $\tau_G \circ G \circ \tau_G^{-1}$ and $F_{AS}$ are also $C^0$-close. Therefore, the map $\zeta$ for $\tau_G \circ G \circ \tau_G^{-1}$ still satisfies condition \eqref{eq:mostly-contracting} of Section~\ref{se:proof-Cor-C} with respect to the reference measure $\nu_\varphi$ (which is intrinsic to the dynamics $\sigma \times L$ and the potential $\varphi$, and does not depend on $G$ nor on $\tau_G$). Thus, Corollary~\ref{cor:maincorollary-1} is also valid for $\tau_G \circ G \circ \tau_G^{-1}$. Consequently, if we fix a basic set $\Lambda$ of $\sigma$, the uniqueness of equilibrium states for the restriction of $\tau_G \circ G \circ \tau_G^{-1}$ to $\Lambda \times \mathbb{T}^2$ may be obtained from Theorem~\ref{teo:maintheorem-2}, and conveyed to $G$ through the conjugacy $\tau_G$.

\section{Proof of Corollary~\ref{cor:maincorollary-2} }\label{se:stability}

Here we will analyze the statistical stability of the equilibrium states of $F_{AS}$. The argument may easily be reformulated for $F_{Sb}$, whose main difference to $F_{AS}$ is the fact that $\Omega(F_{Sh})$ has a unique basic piece, which is a manifold. To simplify the presentation, firstly we will vary the potential, then the dynamics, and only afterwards will we blend these two perturbations.

\begin{lemma}\label{le:mu-zero invariant}
Let $G_n \in \mathcal{V}$ be a sequence of diffeomorphisms converging to $F_{AS}$ in the $C^1$-topology and let $(\eta_n)_{n \,\in \,\mathbb{N}}$ be a sequence of $G_n$-invariant probability measures which converges to a probability measure $\eta$ in the weak$^{*}$-topology. Then $\eta$ is $F_{AS}$-invariant.
\end{lemma}

\begin{proof} We need to show that, given a continuous map $\psi: \Lambda \times \mathbb{T}^2 \, \to \, \mathbb{R}$, then
$$\int \, (\psi \circ F_{AS})\, d\eta = \int \, \psi \, d\eta.$$
As $(\eta_n)_{n \,\in \,\mathbb{N}}$ converges in the weak$^{*}$-topology to $\eta$, then
$$\lim_{n \, \to \, +\infty}\,\,\int \,\psi \, d\eta_n = \int \,\psi \,d\eta.$$
By the $G_n$-invariance of $\eta_n$, we also have
$$\int \,(\psi \circ G_n) \, d\eta_n = \int\,\psi \,d\eta_n, \quad \forall \, n \, \in \,\mathbb{N}.$$
Besides,
$$\lim_{n \, \to \, +\infty}\,\,\int \,(\psi \circ G_n) \, d\eta_n = \int\,(\psi \circ F_{AS}) \,d\eta$$ since
\begin{eqnarray*}
&&\Big|\int \,(\psi \circ G_n) \, d\eta_n - \int \,(\psi \circ F_{AS}) \, d\eta \Big| \\
&\leq& \Big|\int \,(\psi \circ G_n) \, d\eta_n - \int \,(\psi \circ F_{AS}) \, d\eta_n\Big| + \Big|\int \,(\psi \circ F_{AS}) \, d\eta_n - \int \,(\psi \circ F_{AS}) \, d\eta\Big|\\
&\leq& \|(\psi \circ G_n) - (\psi \circ F_{AS})\|_{C^0} + \Big|\int \,(\psi \circ F_{AS}) \, d\eta_n - \int \,(\psi \circ F_{AS}) \, d\eta\Big|
\end{eqnarray*}
and the first term of the previous estimate goes to $0$ due to the uniform continuity of $\psi$ and the $C^0$-convergence of $(G_n)_{n \,\in \,\mathbb{N}}$ to $F$, while the second goes to $0$ because $(\mu_n)_{n \,\in \,\mathbb{N}}$ converges in the weak$^{*}$-topology to $\mu_0$.
\end{proof}

Having fixed a basic set $\Lambda$ in $\Omega(\sigma)$, let us now perturb the potential in $\mathbb{P}(\Lambda \times \mathbb{T}^2)$.

\begin{lemma}\label{le:stat.stab.-2}
Given a basic set $\Lambda$ of $\sigma$ and $\varphi \in \mathbb{P}(\Lambda \times \mathbb{T}^2)$, consider a sequence $(\varphi_n)_{n \, \in \, \mathbb{N}}$ of maps in $\mathbb{P}(\Lambda \times \mathbb{T}^2)$ which converges in the $C^0$-norm to $\varphi$. Take the sequence of unique equilibrium states $\mu_n$ for $F_{AS}$ and $\varphi_n \circ H$, and assume that $(\mu_n)_{n \,\in \,\mathbb{N}}$ converges in the weak$^{*}$-topology to a probability measure $\eta$ in $\Lambda \times \mathbb{T}^2$. Then $\eta$ is the equilibrium state $\mu_{\varphi \circ H}$ of $F_{AS}$ and $\varphi \circ H$.
\end{lemma}

\begin{proof} We already know (cf. Lemma~\ref{le:mu-zero invariant}) that $\eta$ is $F_{AS}$-invariant, so we are left to show that
$$P_{\text{top}}(F_{AS}, \,\varphi \circ H) = h_{\eta}(F_{AS}) + \int\,(\varphi \circ H)\,d\eta.$$
As $F_{AS}$ is entropy-expansive in $\Lambda \times \mathbb{T}^2$, the entropy map of $F_{AS}$ is upper semi-continuous, so
$$\limsup_{n \, \to \, +\infty}\,\,h_{\mu_n}(F_{AS}) \leq  h_{\eta}(F_{AS}).$$
Besides, as $\lim_{n \, \to \, +\infty}\,\mu_n = \eta$ and $\lim_{n \, \to \, +\infty}\,\varphi_n = \varphi$,
$$\lim_{n \, \to \, +\infty}\,\,\int \, (\varphi_n \circ H) \, d\mu_n = \int \, (\varphi \circ H) \, d\eta.$$
So, by \cite[Theorem 9.7]{W1981},
$\lim_{n \, \to \, +\infty}\,\,P_{\text{top}}(F_{AS}, \,\varphi_n \circ H) = P_{\text{top}}(F_{AS}, \,\varphi \circ H)$. Hence
\begin{eqnarray*}
P_{\text{top}}(F_{AS}, \,\varphi \circ H) &=& \lim_{n \, \to \, +\infty}\,\,P_{\text{top}}(F_{AS}, \,\varphi_n \circ H) \\
&=& \limsup_{n \, \to \, +\infty}\,\,\Big[h_{\mu_n}(F_{AS}) + \int\,(\varphi_n \circ H)\,d\mu_n\Big]\\
&=& \limsup_{n \, \to \, +\infty}\,\,h_{\mu_n}(F_{AS}) + \lim_{n \, \to \, +\infty}\,\, \int\,(\varphi_n \circ H)\,d\mu_n\\
&\leq& h_{\eta}(F_{AS}) + \int\,(\varphi \circ H)\,d\eta.
\end{eqnarray*}
Therefore, $\eta$ is an equilibrium state of $F_{AS}$ and $\varphi \circ H$. And, by Corollary~\ref{cor:maincorollary-1}, this is unique.
\end{proof}

Let us now vary simultaneously the diffeomorphism and the potential.

\begin{proposition}
Given a basic set $\Lambda$ of $\sigma$ and a potential $\varphi \in \mathbb{P}(\Lambda \times \mathbb{T}^2)$, consider a sequence $(\varphi_n)_{n \, \in \, \mathbb{N}}$ of maps in $\mathbb{P}(\Lambda \times \mathbb{T}^2)$ converging in the $C^0$-norm to $\varphi$. Let $(G_n)_{n \,\in \,\mathbb{N}}$ be a sequence of diffeomorphisms in $\mathcal{V}$ converging in the $C^1$-topology to $F_{AS}$. Take the unique equilibrium state $\eta_n$ for $G_n$ and the potential $\varphi_n \circ h_{G_n} \circ \tau_{G_n}$, and suppose that the sequence $(\eta_n)_{n \,\in \,\mathbb{N}}$ converges in the weak$^{*}$-topology to a probability measure $\eta$ in $\Lambda \times \mathbb{T}^2$. Then $\eta$ is the equilibrium state $\mu_{\varphi \circ H}$ of $F_{AS}$ and $\varphi \circ H$.
\end{proposition}

\begin{proof} We already know that
$$P_{\text{top}}(G_n, \,\varphi_n \circ h_{G_n} \circ \tau_{G_n}) = P_{\text{top}}(\sigma \times L, \varphi_n) = P_{\text{top}}(F_{AS}, \,\varphi_n \circ H)$$
and
$$\lim_{n \, \to \, +\infty}\,\, P_{\text{top}}(F_{AS}, \,\varphi_n \circ H) = \lim_{n \, \to \, +\infty}\,\,P_{\text{top}}(F_{AS}, \,\varphi \circ H)$$
where the first two equalities come from \eqref{eq:pressures} and the last one is a direct consequence of \cite[Theorem 9.7]{W1981}. Moreover, $\nu_{\varphi_n}=(h_{G_n} \circ \tau_{G_n})_*(\eta_n)$ is the unique equilibrium state of $\sigma \times L$ and $\varphi_n$ (cf. Theorem~\ref{teo:maintheorem-1}); and the sequence $(\nu_{\varphi_n})_{n \, \in \, \mathbb{N}}$ converges to $H_*(\eta)$ since

\begin{lemma}\cite[Lemma 1]{NY1983} and \cite[\S8]{HPS1977} $\,$
\begin{itemize}
\item [(a)] $\lim_{n \, \to \,+ \infty}\,\|\,\tau_{G_n} - \text{identity }\|_{C^0} = 0$.
\item [(b)] $\lim_{n \, \to \,+ \infty}\,\|\,h_{G_n} - H\,\|_{C^0} = 0$.
\end{itemize}
\end{lemma}

\noindent Thus, by Proposition~\ref{prop:stat.stab.}, $H_*(\eta)=\nu_\varphi$, the unique equilibrium state of $\sigma \times L$ and $\varphi$. Besides, by Corollary~\ref{cor:maincorollary-1}, $\mu_\varphi =(H_{|_{B \times \mathbb{T}^2}}^{-1})_*(\nu_\varphi)$ is the unique equilibrium state of $F_{AS}$ and $\varphi \circ H$. Finally,
$$\mu_\varphi =
\left(H_{|_{B \times \mathbb{T}^2}}^{-1}\right)_*(\nu_\varphi) =
\left(H_{|_{B \times \mathbb{T}^2}}^{-1}\right)_* \Big(H_{|_{B \times \mathbb{T}^2}}\Big)_*(\eta) = \eta.$$
\end{proof}


\end{document}